\documentclass[11pt,a4paper]{article}
\usepackage[latin1]{inputenc}
\usepackage[TS1,T1]{fontenc}
\usepackage[english]{babel}
\usepackage{amsfonts,amssymb,latexsym,xspace}
\usepackage{amsmath,enumerate}
\usepackage{amsthm}
\usepackage{amscd}
\usepackage{dsfont}
\usepackage{verbatim}
\usepackage{graphics}
\usepackage{graphicx,pstricks}
\usepackage{amsfonts,latexsym,epsf}
\usepackage[all]{xy}

\numberwithin{equation}{section}

\newcommand{\ds}{\displaystyle}
\newcommand{\p}{\mathcal{P}}
\newcommand{\D}{\mathcal{D}}

\newcommand{\h}{{\check\rho^*}}
\newcommand{\Z}{\mathbb{Z}}
\newcommand{\N}{\mathbb{N}}

\newcommand{\R}{\mathbb{R}}
\newcommand{\T}{\mathcal{T}}

\newcommand{\E}{\mathcal{E}}

\newcommand{\A}{\mathcal{A}}

\newcommand{\rn}{{\R^d}}
\newcommand{\im}{\mbox{\rm im}}
\def \marginpar#1{}
\newcommand{\APG}{Anderson-Putnam-G\"ahler}
\newcommand{\PEPV}{\phi}
\newcommand{\iso}{\psi}
\newcommand{\f}{\eta}

\newcommand{\CC}{\mathcal C}
\renewcommand{\t}{\textsc{t}}

\newtheorem{thm}{Theorem}[section]

\newtheorem{cor}[thm]{Corollary}

\newtheorem{lem}[thm]{Lemma}

\newtheorem{df}[thm]{\bf Definition}
\newtheorem{rem}[thm]{\bf Remark}

\def\pn {\par \noindent}

\title{Comparing different versions of tiling cohomology}

\author{Housem Boulmezaoud and Johannes Kellendonk\\
${}$\\
{\small Universit\'e de Lyon, Universit\'e Claude Bernard Lyon 1,}\\
{\small Institut Camille Jordan, CNRS UMR 5208,}\\ 
{\small 43 boulevard du 11 novembre 1918, F-69622 Villeurbanne cedex}}

\date{\today}

\date{April 6, 2010}
\begin{document}
\maketitle

\begin{abstract} We establish direct isomorphisms between different versions of tiling cohomology. The first version is the direct limit of the cohomologies of the approximants in 
the Anderson-Putnam-G\"{a}hler system, the second is the recently introduced PV-cohomology of Savinien and Bellissard and the third is pattern equivariant cohomology.
For the last two versions one can define weak cohomology groups. We show that the 
isomorphisms extend to the weak versions. 
This leads to an alternative formulation of the pattern equivariant mixed quotient group which describes deformations of the tiling modulo topological conjugacy. 
\end{abstract}

\section{Introduction.}
There exist various isomorphic pictures of what one now calls tiling
cohomology. Originally it was defined as the groupoid
 cohomology of the tiling groupoid \cite{K0}
or as the \v{C}ech cohomology of the tiling space (the hull)
\cite{AP}. If the tiling space can be seen as an inverse limit of finite
simplicial (or cellular) complexes.
the cohomology of the tiling space is
a direct limit of simplicial (or cellular) 
cohomology groups  \cite{AP}, \cite{G,S1}.
An intuitively simpler version was later proposed using
 the concept of pattern equivariant exterior forms \cite{K1,KP} and then
 modified using pattern equivariant simplicial (or cellular)
cochains \cite{S2}. Finally, a
 recent version making full use of the simplicial structure 
 defined by the tiling goes under the name of PV-cohomology
 \cite{BS}. Each version has his own advantages and disadvantages.
 Computationally the groupoid cohomology approach was most successful
 for cut \& and project patterns \cite{FHKmem} 
and the direct limit approach 
for substitution tilings
\cite{AP}. 
The newer versions led conceptionally to new insight, especially the pattern
 equivariant de Rham cohomology. The latter comes about as it
 naturally offers different choices of cohomology groups
based on the distinction between algebraic modules and their
topological closure. 
The original tiling cohomology is isomorphic to strongly pattern
equivariant cohomology, the algebraic version. 
The topologically closed version is the so-called weakly pattern
equivariant cohomology, and this one is isomorphic to 
the tangential cohomology of the hull (seen as a lamination)
\cite{KP}. It has been shown that the theory of deformations of
tilings is related to 
the mixed group of strongly pattern equivariant cochains modulo weakly
pattern equivariant coboundaries \cite{K2}. It is therefore of quite 
some interest to determine the latter. However, apart from
substitution tilings \cite{CS}, we have at present no general method
to offer to compute this mixed group. 

The present article sheds new light on the different versions of
pattern equivariant cohomology in that it provides analogues of the
weak and the mixed version in the framework of PV-cohomology.
To obtain these we provide explicit descriptions of the isomorphisms
between the various cohomology groups. We
establish in particular an explicit formula for the
isomorphism between pattern equivariant de Rham cohomology and
PV-cohomology and show that it carries the right continuity properties
for it to extend from the algebraic to the topological closed
versions of the cohomology groups.

\section{Preliminaries.}
In this section we provide some known background material on tilings.
In particular we recall the definitions of the following
versions of tiling cohomology: \begin{enumerate}
\item as direct limit of the simplicial cohomologies of the 
Anderson-Putnam-G\"ahler complexes 
  \cite{AP,G,S1} (we refer to that simply as DL-cohomology), 
\item as PV cohomology \cite{BS},
\item as PE de Rham cohomology \cite{K1,KP},
\item as PE simplicial cohomology \cite{S2}.
\end{enumerate}
We also recall or introduce weak versions of PV and PE cohomology.
PE stands for {\em pattern equivariant}, s-PE for {\em strongly} and
w-PE for {\em weakly}  pattern equivariant.
\subsection{Tilings and tilings spaces}
A tiling of the euclidean space $\R^d$ is a covering of the space by
closed topological disks (its tiles) which overlap at most on their
boundaries. We consider here simplicial tilings 
of  (translational) finite local complexity. 
This means that the tiles are simplices of $\R^d$ which touch face to
face and that there are only finitely many simplices up to
translation; we present more details in Section~\ref{sec-simplicial}. 
In the context of tilings of  finite local complexity the restriction
to simplicial tilings is topologically not a restriction, since there
are ways to associate to any tiling of finite local complexity a
simplicial tiling which carries the same topological information.


The tiling space or hull  of a tiling $\p$ is
$$\Omega_\p:={\overline{\{\t^x(\p)|x\in\R^d\}}^D},$$
the completion of the translation orbit of $\p$ w.r.t\ the metric given by
$$D(\p_1,\p_2)=\inf_{r>0}\{\frac{1}{r+1}|\exists x,y\in
B_\frac{1}{r}(0):B_r[\t^x(\p_1)]=B_r[\t^y(\p_2)]\}$$ 
where $\p_1$ and $\p_2$ are tilings of $\R^d$.
Here $\t^x$ denotes the translation action $\t^x(Q) = Q-x$ where $Q$
is a point or a geometric object and
we use the notation $B_r[\p]=B_r[\p']$ to indicate that
the tiling $\p$ and $\p'$ coincide\footnote{e.g.\ their sets of  
 boundary points of their tiles coincides on the $r$-ball around $0$}
on the $r$-ball around $0$.

Tiling spaces can be obtained as inverse limits of simpler spaces.
Such a construction was given for substitution tilings by \cite {AP}
and then later generalized to all tilings of finite local complexity
in \cite{G,S1} and a somewhat similar construction given in \cite{BBG}.
We present the details of the construction using $k$-neighborhoods.

Let $\p$ be a tiling of $\R^d$ and $t$
be the tile (or the patch) of $\p$. The
first corona of $t$ is the set of its nearest neighbors, the second
corona of $t$ is the set of the second nearest neighbors, and so on
out to the $k^{th}$ corona. The $k$-neighborhood of a tile is the collection of its $j$th coronae for 
$j\leq k$.
Consider two tiles $t_1$,
$t_2$ in $\p$, we say that $t_1$ is equivalent to $t_2$, if $t_1$ is
translationally congruent to $t_2$, the equivalence class is called
prototile. The condition of finite local complexity implies that
there are finitely many prototiles. A $k$-collared tile is a tile
labelled with its $k$-neighborhood. The labelling of the tile does
not change the set covered by the tile, which we call its support, but it is only a decoration of it.
A $k$-collared prototile is the translational
congruence class of the $k$-collared tile. Again, finite local
complexity implies that the number of $k$-collared prototiles is
finite. Let $f_1$ and $f_2$ to be faces of two $k$-collared
prototiles $\tilde{t}_1$ and $\tilde{t}_2$. We say $f_1$ is
equivalent to $f_2$ if $\tilde{t}_1$ has a representative $t_1$ and
$\tilde{t}_2$ has a representative $t_2$ in $\p$ such that the edges
corresponding to $f_1$ and $f_2$ coincide. The disjoint union of the
supports of the $k$-collared prototiles quotiened out by this
equivalence relation on the faces is the set $\Gamma_k$. Equipped with
the quotient topology we consider it as a topological space. It
inherits the structure of a complex from $\p$ and is therefore called an
\APG\ complex.   

Let
$\rho_{k+1\,k}:\Gamma_{k+1}\longrightarrow\Gamma_k$ the map which
associates to a point 
of a $k+1$-collared prototile the same point in the $k$-collared
prototile obtained by forgetting about the $(k+1)^{th}$ corona in its label. 
The inverse system
$(\Gamma_k,\rho_{k \,k-1})_{k\in\N}$ is called the
Anderson-Putnam-G\"{a}hler system.  
Its importance lies in the fundamental result that its inverse limit
is homeomorphic to the tiling space
$$\Omega_\p \cong \lim\limits_{\substack{ \longleftarrow \\ k}}{\Gamma_k}.$$
A point of $\Gamma_k$ tells us how to place a patch corresponding to
a $k$-collared tile around the origin of $\R^d$. By definition of the
inverse limit the surjections $\rho_{k \,k-1}$ induce a continuous surjection
\begin{equation}\label{gamma}
\rho_k:\Omega_\p\longrightarrow \Gamma_k
\end{equation}
which can be described as follows.
Consider the tile of $\omega\in\Omega_\p$ together with its $k$-neighborhood
which lies on the origin $0_{\R^d}$. Taking congruence
classes w.r.t.\ translations, the point in the tile above corresponds
to a point in a $k$-collared prototile. $\rho_k(\omega)$ is this
point viewed as a point in $\Gamma_k$. By the definition of the
equivalence for $\Gamma_k$, this is well defined also in the cases in
which a boundary 
point of a tile lies on the origin.
Related to the above is the surjection
$\tilde\rho_k:\R^d\longrightarrow \Gamma_k$
\begin{equation}\label{eq-trho}
\tilde\rho_k(x) = \rho_k(\t^x(\p)).
\end{equation}
\subsubsection{Simplicial tilings and $\Delta$-complexes}\label{sec-simplicial} 
For their definition of PV-cohomology Savinien and Bellissard consider
tilings whose tiles are finite $\Delta$-complexes which are compatible
in the sense that the intersection of two tiles are
sub-$\Delta$-complexes. Intuitively the reader should simply have in
mind a tiling
whose tiles are simplices which meet face to face and such that
simplices can be pattern equivariantly oriented. 
We call such tilings simplicial tilings.

In more technical tems, a $\Delta$-complex structure (see \cite{H})
on a topological space $X$ is a collection of continuous maps 
$\sigma:\Delta^n\to X$, where
$$\Delta^n=\{(x_1,x_2,\cdots,x_{n+1})\in\R^{n+1} | \sum_{i=1}^{n+1} x_i =
1, x_i\geq 0\}$$  
is the standard $n$-simplex in euclidean space. It is supposed that
these maps are injective on the interior of $\Delta^n$.
They are referred to as the characteristic maps and their images are the so-called cells (of dimension $n$) of $X$.
It is also supposed that each point of $X$ lies in the interior of exactly one cell.
$\Delta^n$ has $n+1$ faces of dimension $n-1$, we write
$\partial_i\Delta^n$ for the face described by coordinates
$(x_1,x_2,\cdots,x_{n+1})$ for which $x_i=0$ and call it the $i$th
face. We write $\partial_i\sigma$ for the
restriction of $\sigma$ to $\partial_i\Delta^n$. 
It is also a characteristic map,
but for an $n-1$ cell.
When cells are next to each other, touching along lower dimensional faces, 
or when characteristic functions are not injective on the boundary but identify different faces,
then the corresponding characteristic maps
of the lower dimensional faces are identified.
The vertices of $\Delta^n$ are ordered, the $i$th
vertex being the one having coordinate $x_i=1$, and this order passes
to the faces. It is assumed that the identification
of characteristic maps 
preserves the ordering of the vertices.
They are further compatibility conditions among the
characteristic maps for which we refer to \cite{H,BS}.
We will use the phrasing "the simplex $\sigma$ of $X$" 
to refer to the characteristic map or to the cell it defines in $X$.

A simplicial complex is special kind of $\Delta$-complex. In \cite{ST}
this notion is used for $\Delta$-complexes discribing subsets of 
euclidean  spaces whose cells are actually simplices, i.e.\ convex hulls of
$n+1$ vertices (in general position) if $n$ is their dimension.

\begin{lem}\label{lem-simpl}
Any tiling of finite local complexity is mutually locally derivable
with a tilings whose tiles are the top-dimensional simplices of a
simplicial complex. Moreover, the simplicial structure is pattern equivariant
in the following sense. 
If $t_1$ and $t_2$ are tiles which differ by a translation, i.e.\ $t_2 = \t^x(t_1)$ for some $x$, 
then the characteristic functions $\sigma_1$ and
$\sigma_2$ of $t_1$ and $t_2$, respectively, differ by the same translation, 
$\sigma_2=\t^x\circ\sigma_1$. 
\end{lem}

\pn\textbf{\emph{Proof.}}
It is well known that any tiling of finite local complexity is
mutually locally derivable with a tiling whose tiles are polyhedra and
match face to face. One may then subdivide the polyhedra into
simplices in a mutually locally derivable way
to obtain a tiling whose tiles are simplices
and match face to face. The issue is now to order the vertices of 
the simplices in a
pattern equivariant way. This can be done as follows: By the finite
local complexity the exists a unit vector $\nu$ in $\R^d$ (the ambient
space) such that none of the edges of the
tiles is perpendicular to $\nu$. 
Then the order on the vertices defined by $v_1\leq
v_2$ provided $\langle\nu,v_1\rangle\leq \langle\nu,v_2\rangle$ is
total and translation invariant in the sense that the pair
$(v_1-x,v_2-x)$ (if it occurs in the tiling)
has the same order than $(v_1,v_2)$.  
We may therefore define the characteristic maps to be translation
invariant in the same sense.
\qed\bigskip

We call a tiling with this property a simplicial tiling.
We denote by $\p^{n}$ the
$n$-skeleton of the simplicial complex it defines,
that is, $\p^{n}$ 
is the union of all faces of dimension smaller or equal to $n$. In
particular $\p^{d}=\R^d$. 

If $X$ carries in addition the structure of a differentiable manifold
and we want to integrate $n$-forms over certain $n$-dimensional
submanifolds it is useful to use charts defined by the cells.  
For that purpose we shall not only require the characteristic maps to
be diffeomorphisms, but even that they extend to diffeomorphisms from
an open neighborhood of $\Delta^n$ to an open neighborhood of the
corresponding cell;  
in the terminologie of \cite{ST} this means that the simplicial
complex defines a smooth triangulation of $X$.  
This is clearly possible for the simplicial complex defined by a
tiling in $\R^d$. 

The tiling defines a $\Delta$-complex for 
the topological spaces $X=\Gamma_k$. This is obvious from the
construction of $\Gamma_k$. In fact, if $\sigma$ is a characteristic
map of a $k$-collared tile then
$\tilde\sigma:=\tilde\rho_k\circ\sigma$ is a characteristic map for
the corresponding $k$-collared prototile in $\Gamma_k$. By 
Lemma~\ref{lem-simpl}
this is independent of the choice of representative for the
prototile. Note that $\tilde\sigma$ is not necessarily injective.
We denote by
$\mathcal{S}_k^n$ the set of characteristic maps of 
the $n$-simplices on $\Gamma_k$.

An $n$-chain on $X$ is a formal linear combination of oriented simplices from
$X$, and we denote the set of $n$-chains by $C_n(X)$. It is thus the
free $\Z$-module generated by elements $\langle\sigma\rangle$.
An $n$-cochain with coefficients in an abelian group $A$ is an element
of the $\Z$-module $C^n(X,A)={\rm Hom}_\Z(C_n(X),A)$. 
Let $a\in A$. We denote by $a_\sigma$ the morphism defined by 
$$a_\sigma(\langle\tau\rangle) =  \left\{\begin{array}{cl} 
a & \mbox{if } \sigma = \tau\\
0 & \mbox{else}
\end{array}\right.$$.

\subsubsection{The $\Delta$-transversal of a tiling}\label{sec-trans}

The canonical transversal or discrete hull of a tiling is defined with
the help of punctures for the tiles of the tiling \cite{Ke1}. In
\cite{BS} punctures for the faces of tiles are introduced as well so
that one gets a complex defined by transversals of different dimension.

A puncture of a face of a tile is a chosen point in the face, for instance
its barycenter. The pair given by the face and its puncture is called a
punctured face. 
All punctures are supposed to be chosen in such way that two faces of
tiles in the tiling which agree up to translation have their punctures
at the same relative position.  
Then the puncturing passes to the equivalence classes so that the
simplices of $\Gamma_k$ are also punctured. 
We denote by $\p^{n,punc}$ the set of punctures of all $n$-faces of $\p$.
\begin{df}\cite{BS}
 The $n$-dimensional $\Delta$-transversal is
 $$\Xi^n=\{\p' \in\Omega_\p|0\in{\p'}^{n,punc}\}.$$
The disjoint union 
$\Xi=\dot\bigcup_{n=0}^d \Xi^n$ is called the $\Delta$-transversal of
the tiling $\p$. 
\end{df}
In \cite{BS} the notation $\Xi_\Delta$ for the $\Delta$-transversal 
was used instead which distinguishes it from the usual notation for
the canonical transversal, but we see no danger of confusion here.
The sets $\Xi^n$ are totally disconnected and compact subsets
of $\Omega_\p$. 
Their topology is generated by the following clopen subsets, referred to as
acceptance zones for faces of $\Gamma_k$.
\begin{df}\label{az} The acceptance zone of an $n$-simplex $\sigma$ of
  $\Gamma_k$ is the subset $\Xi(\sigma)$ of $\Xi^n$ given by
$$\Xi(\sigma)=\rho_k^{-1}(p_\sigma)$$
where $p_\sigma$ is the puncture of $\sigma$.
The surjections $\rho_k$ are defined in (\ref{gamma}).
\end{df}
\begin{rem}
Recall that a $d$-simplex $\sigma$ in $\Gamma_k$ corresponds to (the
congruence class of) a tile 
$t_\sigma$ in $\p$ labelled with its $k$-neighborhood. In a similar fashion a $n$-simplex
$\sigma$ in $\Gamma_k$ corresponds to an $n$-dimensional
face  of a tile in $\p$ together with its $k$-neighborhood. Then
$\Xi(\sigma)$ can be 
seen as the subset of tilings which contain a patch corresponding to a
$k$-collared face defined by $\sigma$ whose puncture lies on $0\in\R^d$.
\end{rem}

We denote the restriction of $D$ to $\Xi$ by $D_0$. The family of clopen balls
$\{\mathcal{B}_{D_0}(\mathcal{P},\varepsilon)\}_{{\mathcal{P}\in\Xi}}$,
 of radius $\varepsilon$ around $\p$, is a base of
the topology of the space $(\Xi, D_0)$. We denote
$\mathcal{B}_{D_0}^n(\p,\varepsilon) =
\mathcal{B}_{D_0}(\p,\varepsilon)\cap \Xi^n$.

\begin{lem}\label{lem-topo}
The family of acceptance zones
$\{\Xi(\sigma)\}_{\sigma\in\Gamma_k,0\leq k\leq d}$ is
a base for the topology of the metric space
$(\Xi,D_0)$.
\end{lem}

\pn\textbf{\emph{Proof.}}
Let $\sigma$ be an $n$-simplex of $\Gamma_k$ and $\T$ a tiling
in $\Xi^n(\sigma)$. Recall that $D_0(\T,\T')\leq
\frac{1}{R+1}$ means that $\T$ and $\T'$ agree on the ball $B_R(0)$ of $\R^d$
of radius $R$ around the origin. Choose $R>0$ large enough such that 
$B_R(0)$ covers the $k$-neighborhood of the face of $\T$ at the origin. Then
$D_0(\T,\T')\leq \frac{1}{R+1}$ implies that $\T'$ contains as well that face with its $k$-neighborhood
and hence belongs to the acceptance zone $\Xi(\sigma)$. We thus have proved $\exists R>0$:
$\T\in\mathcal{B}_{D_0}^n(\T,\frac{1}{R+1})\subset\Xi(\sigma)$.

\pn Conversely, let $\T$ be a tiling of $\Xi^n$
 and $R>0$. 
Choose $k$ large enough so that the $k$-neighborhood of the face $f$ at the origin of $\T$ covers $B_R(0)$. Let $\sigma_f\in\Gamma_k$ be the $n$-simplex corresponding to this $k$-collared face.
Then $\T'\in \Xi(\sigma_f)$ implies that
$\T$ and $\T'$ agree on $B_R(0)$ and hence $D_0(\T,\T')\leq
\frac{1}{R+1}$. We thus have proved that  there exists an
acceptance zone $\Xi(\sigma_f)$ such that
$\T\in\Xi(\sigma_f)\subset\mathcal{B}_{D_0}^n(\T,\frac{1}{R+1})$.
\qed\bigskip


We denote by $\CC_{lc}(\Xi^n,A)$ the
$\Z$-module of $A$-valued locally constant functions over
the $n$-dimensional $\Delta$-transversal. 
It constitutes the degree $n$ part of the 
complex 
defining PV cohomology which we will consider below.

It follows from Lemma~\ref{lem-topo} that any element of
$\CC_{lc}(\Xi^n,A)$ is a finite sum of elements 
$a_{\Xi(\sigma)}$ where $a\in A$ and 
$$ a_{\Xi(\sigma)}(\xi) = \left\{\begin{array}{cl}
a & \mbox{if } \xi\in \Xi(\sigma)\\
0 & \mbox{else}
\end{array}\right. .$$


\subsection{Cohomology as a direct limit (DL-cohomology)}

As mentioned in the introduction, one way to define the cohomology of
a tiling $\p$ is as \v{C}ech cohomology $\check{H}(\Omega_\p,A)$
of $\Omega_\p$. 
Since $\Omega_\p$ is an inverse limit of spaces
$(\Gamma_k)_{k\in\N}$ its \v{C}ech cohomology can be expressed as a
direct limit of groups
$$\check H(\Omega_\p,A)=\check H(\lim\limits_{\substack{
    \longleftarrow \\ l}}{\Gamma_l},A)=\lim\limits_{\substack{
    \longrightarrow \\ l}}{\check H(\Gamma_l,A)}.$$
The main advantage is that $\check H(\Gamma_k,A)$, the \v{C}ech
cohomology of  $\Gamma_k$, is simple to compute. 
As $\Gamma_k$ is a finite $\Delta$-complex its \v{C}ech
cohomology is isomorphic to its simplicial cohomology.
Furthermore, the maps $\rho_{k\,k-1}$ preserve the $\Delta$-complex
structure and so 
the maps they induce on simplicial cohomology can easily be handled. 
In its original version this was the method to calculate the tiling
cohomology for substitution tilings in \cite{AP} (although they used
cellular cohomology). 

\subsection{Savinien and Bellissard's approach to tiling
  cohomology}\label{sec-2.6} 

\subsubsection{PV-cohomology.}

Recall that $\mathcal{S}_0^n$ denotes the (finite) set of characteristic maps of 
the $n$-simplices in $\Gamma_0$.
Let $\tilde\sigma\in S^n_0$ and
choose a characteristic map $\sigma$ of an $n$-simplex of $\p$ such that
$\tilde\sigma = \tilde\rho_0\circ\sigma$.
We define the vector 
$$x_{\tilde\sigma,i} =p_{ \partial_i\sigma}-
p_\sigma\in\R^d$$
which is the difference vector pointing from the puncture of the
cell defined by $\sigma$ to the puncture of its $i$th face.
This is independent of the choice of $\sigma$ by Lemma~\ref{lem-simpl}. 
We denote by $\t_{\tilde\sigma,i}:\Xi(\tilde\sigma)\to 
\Xi(\partial_i\tilde\sigma)$ the restriction of
$\t^{x_{\tilde\sigma ,i}}$ to $\Xi(\tilde\sigma)$,
$\t^{x_{\tilde\sigma ,i}}(\xi) = \xi + p_\sigma-p_{ \partial_i\sigma}$.
The following definition is essentially that of \cite{BS}.
\begin{df}
Let $\tilde\sigma\in S^n_0$.
Define the linear map $\theta_{\tilde\sigma,i}:
C_{lc}(\Xi^{n-1},A)\to C_{lc}(\Xi^n,A)$ by  
$$\theta_{\tilde\sigma,i}(f) = 
\imath\circ\t_{\tilde\sigma,i}^*
(f\left|_{\Xi(\partial_i\tilde\sigma)}\right.)$$
where $ \t_{\tilde\sigma,i}^*:C_{lc}(\Xi(\partial_i\tilde\sigma),A)\to C_{lc}(\Xi(\tilde\sigma),A)$ is the pull 
back of $\t_{\tilde\sigma,i}$ and
$\imath:C_{lc}(\Xi(\tilde\sigma),A)\hookrightarrow C_{lc}(\Xi^{n},A)$
the inclusion. The strong PV-cohomology of $\p$ is the cohomology of the
complex $(C_{lc}(\Xi^*,A), d_{PV})$ where
$d_{PV}:C_{lc}(\Xi^{n},A)\longrightarrow
C_{lc}(\Xi^{n+1},A)$ 
is given by
$$ d_{PV}=
\ds\sum_{\tilde\sigma\in\mathcal{S}_0^{n+1}}\sum_{i=0}^{n+1} (-1)^i 
\theta_{\tilde\sigma,i}.
$$
We denote the strong PV-cohomology with coefficients in $A$ by
$H_{PV}^*(\Gamma_0,\CC_{lc}(\Xi,A))$. 
\end{df}

$H_{PV}^*(\Gamma_0,\CC_{lc}(\Xi,\Z))$
is isomorphic to the \v{C}ech cohomology of the hull \cite{BS}. 
\subsubsection{Relation with pattern groupoid cohomology}
\newcommand{\GP}{\mathcal G\p}
PV-cohomology can effectively seen as the continuous cocycle cohomology (after \cite{Ren}) of the pattern groupoid $\mathcal G\p$, which is the variant of the discrete tiling groupoid defined in \cite{FHKmem}. More precisely, the PV-complex is a reduction of the complex of continuous cocycle cohomology which has the advantage of being trivial in degrees larger than the dimension of the tiling. This is similar to using a finite resolution of length $d$ instead of the standard resolution for the cohomology of the group $\Z^d$.
Let us explain this.
   
By definition,
$\mathcal G\p$ is the reduction to $\Xi^{(0)}$ of the
transformation groupoid  $\Omega_\p\rtimes \R^d$. Its elements are pairs $(\omega,v)\in \Omega_\p\times \R^d$ which satisfy $\omega,\t^v(\omega)\in \Xi^{(0)}$. 
The continuous cocycle cohomology of $\mathcal G\p$ is the cohomology of the complex of continuous functions
$$ C(\GP^{(n)},\Z)\stackrel{\delta}{\longrightarrow} C(\GP^{(n+1)},\Z) \cdots $$
where $\GP^{(0)}=\Xi^{0)}$ and $\GP^{(n)}$ is the subset of elements $(\omega,v_1,\cdots,v_n)\in \Omega_\p\times(\R^d)^n$ which satisfy $\omega\in\Xi^{(0)}$ and $\omega-\sum_{i=1}^j v_i \in\Xi^{(0)}$ for all $j$. Its topology is the subspace topology of the product topology.  
The differential is given by $\delta f (\omega,v) = f(\t^v(\omega))-f(\omega)$ 
in degree $0$, and in degree $n>0$ by
\begin{eqnarray*}
\delta f (\omega,v_1,\cdots,v_{n+1}) & = &  f(\t^{v_1}(\omega),v_2,\cdots,v_{n+1})+ (-1)^{n+1}f (\omega,v_1,\cdots,v_{n})\\
& &  +\sum_{i=1}^n (-1)^i f(\omega,v_1,\cdots,v_i+v_{i+1},\cdots,v_{n+1}) 
\end{eqnarray*}
We define the map
$\alpha : \Xi^{(n)} \to \GP^{(n)}$ by 
$$ \alpha(\omega) = (\omega,\sigma(x_1)-\sigma(x_0),\cdots,\sigma(x_n)-\sigma(x_{n-1}))$$
where $\sigma(x_i)$ is the $i$th vertex of the $n$-simplex of $\omega$ whose puncture is on $0$.  
Clearly $\alpha$ is injective. It is continuous as the
$\sigma(x_i)-\sigma(x_{i-1})$ are locally defined.
A straightforward computation shows that the pullback of $\alpha$ satisfies
$\alpha^*\circ\delta = d_{PV}\circ \alpha^*$.
It hence yields a surjective chain map. It thus induces a surjective
map from the continuous cocycle cohomology of the pattern groupoid to
PV-cohomology. As we know that the two cohomologies are isomorphic,
$\alpha^*$ has to be an isomorphism at least in the case that the
cohomology is finitely generated.
\subsubsection{Weak PV-cohomology}

We now consider the situation in which the abelian group $A$ 
carries a metric $\delta$ w.r.t.\ which it is complete. Then
$\CC(\Xi^n,A)$ is a complete module w.r.t.\ the metric $\tilde\delta(f,g):=
\sup_{\xi\in\Xi^n}\delta(f(\xi),g(\xi))$.   
\begin{lem}  The $PV$-differential extends to a continuous map
  $d_{PV}:\CC(\Xi^n,A)\stackrel{}\rightarrow \CC(\Xi^{n+1},A)$. 
\end{lem}
\pn\textbf{\emph{Proof.}} 
First observe that $\t_{\tilde\sigma,i}$ is a partial translation.
$\theta_{\tilde\sigma,i}$ is essentially its pull-back and since
domain and range of $\t_{\tilde\sigma,i}$ are clopen
$\theta_{\tilde\sigma,i}(f)$ is continuous for continuous $f$. 
Since $\tilde\delta$ is translation invariant the pull back of a
partial translation is a partial isometry. Hence
$\theta_{\tilde\sigma,i}$ is
continuous. Since $d_{PV}$ is a finite sum of
$\theta_{\tilde\sigma,i}$ the result follows.  \qed\bigskip

\begin{df}
The weak $PV$-cohomology is the cohomology of the so-called weak
$PV$-complex $(\CC(\Xi^*,A),d_{PV})$. 
It is denoted by $H^*_{PV}(\Gamma_0,\CC(\Xi,A))$. 
\end{df}
While for discrete groups, like $\Z$, the strong and the weak PV-complex coincide
they differ vastly in the case that $A$ is a continuous group, like
$\R$. Although the inclusion 
$\CC_{lc}(\Xi^n,A)\subset \CC(\Xi^n,A)$ is a chain map,  
there is no evident relation between the two cohomologies in the
continuous case.

\subsection{PE de Rham cohomology}

In this section, we recall the definition of the pattern equivariant de Rham
cohomology of a tiling.

\begin{df} Let $\p$ be a tiling of finite local complexity.
A function $f$ on $\R^d$ is pattern equivariant with range $R>0$ if
$$\forall x,y\in\R^d,\mbox{ }
B_R[\t^x(\p)]=B_R[\t^y(\p)] \Rightarrow f(x)=f(y). $$ 
It is called strongly pattern equivariant (s-PE) if it is pattern
equivariant with some finite range $R$.\footnote{This definition does not 
only apply to tilings but to any kind of pattern of $\R^d$, but it
captures what we want only in the case of finite local complexity.} 
\end{df}
We denote by $\A_{b}^n(\R^d)$ the space of bounded (smooth) $n$-forms
over $\R^d$. Such forms can be seen as smooth functions from $\R^d$ 
into $\Lambda^n{\R^d}^\ast$, the exterior algebra of the dual
of $\R^d$ (equipped with some norm).
We may therefore consider $\A_{s-\p}^n(\R^d)$ the space of s-PE
$n$-forms over $\R^d$. 
Together with the standard exterior derivative $(\A_{s-\p}^*(\R^d),d)$
is a subcomplex of the usual de Rham complex of differential forms on $\R^d$.  

The cohomology of the differential complex $(\A_{s-\p}^*(\R^d),d)$ is
called the strongly pattern equivariant cohomology and denoted by
$H_{s-\p}^*(\R^d)$. As ordinary de Rham cohomology it is  
a cohomology with real coefficients. 
$H_{s-\p}^*(\R^d)$ is isomorphic to 
the \v{C}ech cohomology
of the tiling space with real coefficients \cite{KP}.

Strongly pattern equivariant functions are in some sense algebraic,
namely the size $R$ of the patch around a point $x$ which has to be
inspected to obtain the value of the function is fixed and finite.
This naturally calls for looking at functions which are obtained as
limits of strongly pattern equivariant functions. These are called
weakly pattern-equivariant.
More precisely, the space of smooth weakly pattern equivariant
$n$-forms  $\A_{w-\p}^*(\R^d)$ 
is the closure of $\A_{s-\p}^*(\R^d)$
in $\A_{b}^*(\R^d)$ with respect to the
Fr\'{e}chet topology given by the family of semi norms $s_k$,
$$s_k(f)=\sup_{|\kappa|\leq k}||D^\kappa f||_\infty$$ \pn
where $\kappa=(\kappa_1,\ldots,\kappa_n)\in\N^n$,
$|\kappa|=\sum_{i=1}^n \kappa_i$, $D^\kappa f=\frac{\partial
^{\kappa_1}}{\partial x_1^{\kappa_1}}\cdots \frac{\partial
^{\kappa_n}}{\partial x_n^{\kappa_n}}f$ and $||.||_\infty$ is the norm
of the uniform convergence. 
The weakly pattern equivariant cohomology $H_{w-\p}^*(\R^d,\R)$ 
is the cohomology of the differential complex
$(\A_{w-\p}^*(\R^d),d)$. 
$H_{w-\p}^\ast(\R^d,\R)$ is isomorphic to the tangential cohomology
of the tiling space \cite{KP}. We shall provide below an example which
confirms our expectation that the weak pattern equivariant cohomology
of aperiodic tilings is always infinitely generated. 

In the last section, we will be interested in deformations of tilings.
These deformations are parameterized\footnote{Strictly speaking the
  parameterization is in terms of this group with coefficients in
  $\R^d$.}   
(in \cite{K2}) by a group called
the pattern equivariant mixed group. It is defined by the quotient
$$Z_{s-\p}^1(\R^d)/B_{w-\p}^1(\R^d)\cap Z_{s-\p}^1(\R^d)$$
of closed strongly PE $1$ forms modulo those $1$ forms which are weakly exact,
$Z_{s-\p}^1(\R^d) = \A_{s-\p}^1(\R^d)\cap \ker d$ and $B_{w-\p}^1(\R^d) = d (\A_{w-\p}^0(\R^d))$.

\subsection{PE simplicial cohomology}

Since $\p$ is a simplicial tiling it defines a simplicial complex in $\R^d$.
The concept of pattern equivariance can equally well be defined for
cochains (with values in any abelian group $A$) over this complex. In
fact, a cochain is s-PE if there is some $k>0$ such that 
its value on $\langle\sigma\rangle$ depends only on the $k$-neighborhood of the
simplex $\sigma$ of $\p$. We denote the degree $n$ s-PE cochains with coefficients
in $A$ by $C_{s-\p}^n(\p,A)$.
Clearly
the coboundary operator maps s-PE cochains to s-PE cochains and so we may
consider what is called the s-PE simplicial cohomology, which is the
simplicial cohomology of the complex $(C_{s-\p}^*(\p,A),d_S)$.
We denote it by $H_{s-\p}^*(\p,A)$.

It has been shown in \cite{S2} that s-PE simplicial
cohomology is isomorphic to the \v{C}ech cohomology of the tiling space. 

Just as one can define weak PE de Rham cohomology by using a
completion w.r.t.\ a natural family of semi-norms, one can define weak
PE simplicial cohomology by means of completion in the (metric)
topology of uniform convergence. 
More precisely we define the w-PE cochains $C_{w-\p}(\p,A)$ (with coefficients
in $A$) as the closure of $C_{s-\p}(\p,A)$
in the set of bounded cochains $C_b(\p,A)$  
w.r.t.\ to the metric $\tilde\delta(c,c')=\sup_\sigma \delta(c(\langle\sigma\rangle),c'(\langle\sigma\rangle))$. Clearly the simplicial coboundary-boundary
operator is continuous in the topology and defines the w-PE
complex $(C_{w-\p}^*(\p,A),d_S)$ whose cohomology we call the w-PE
simplicial cohomology and denote by $H_{w-\p}^*(\p,A)$.

\section{PE simplicial cohomology and PV cohomology}
We start with comparing the strong versions of PE simplicial cohomology
and PV cohomology. This material is essentially known though not
written in this form. At the end of the section we compare the weak
versions. This is new. 
\subsection{PV cohomology as DL cohomology}
We need to discuss in some detail the construction of the
isomorphism between DL cohomology and PV cohomology. 
It essentially establishes that PV-cochains yield a model for the
direct limit of the simplicial cochains of the APG complexes. 

We denote in this section the pull back of 
$\rho_{k+1\,  k}:\Gamma_{k+1}\to\Gamma_{k} $ by $\pi_{k\,k+1}$.
Recall that the direct limit of the directed system 
$C^n(\Gamma_k,A)\stackrel{\pi_{k\,k+1}}{\longrightarrow}
C^n(\Gamma_{k+1},A) $ is 
a module (denoted  
$\lim\limits_{\longrightarrow}{C^n(\Gamma_k,A)}$) together with module maps $\pi_k$ such that
the diagram
$$ \xymatrix{ C^n(\Gamma_k,A) \ar[rr]^{\pi_{k\,k+1}= \rho^*_{k+1\,k}} 
\ar[rd]^{\pi_k} &&
  C^n(\Gamma_{k+1},A) \ar[ld]_{\pi_{k+1}} \\ & 
\lim\limits_{\substack{ \longrightarrow{} \\ l}}{C^n(\Gamma_l,A)} }
$$
commutes. It has moreover the universal property that,
whenever another module $M$ with module maps $\pi'_k:C^n(\Gamma_k,A)\to M$  
yields a commuting diagram 
$$ \xymatrix{
            C^n(\Gamma_k,A) \ar@/_/[rd]_{\pi_k'} \ar[rr]^{\pi_{k\,k+1}}
            && C^n(\Gamma_{k+1},A) \ar@/^/[ld]^{\pi_{k+1}'}  \\
             & M } $$
then the $\pi'_i$ factor uniquely through the direct limit \cite{L}. 
The latter means that there exists a
unique homomorphism 
$h:\lim\limits_{\longrightarrow}{C^n(\Gamma_k,A)} \longrightarrow M$ satisfying
$ h\circ \pi_k = \pi'_k$ for all $k\in \N$.

Define the module maps $\f_k:C^n(\Gamma_k,A)\longrightarrow
\CC_{lc}(\Xi_{\Delta}^n,A)$ by
$$\f_k(a_\sigma)=a_{\Xi(\sigma)}.$$
These maps 
induce a new commuting diagram
$$ \xymatrix{
            C^n(\Gamma_k,A) \ar@/_/[rd]_{\f_k} \ar[rr]^{\pi_{kj}}
            && C^n(\Gamma_j,A) \ar@/^/[ld]^{\f_j}  \\
             & \CC_{lc}(\Xi^n,A)} $$
By the universal property of the direct limit there exists a
unique homomorphism $$\f:\lim\limits_{\substack{ \longrightarrow{} \\
l}}{C^n(\Gamma_l,A)} \longrightarrow \CC_{lc}(\Xi^n,A)$$ such that
$$ \f\circ \pi_k = \f_k\quad k\in \N.$$

\begin{lem}\label{inj} The maps $\f_k$ are one-to-one.\end{lem}
\pn \textbf{\emph{Proof.}} 
Let $\xi$ in $\Xi(\sigma_0)$ for some $\sigma_0\in
S_k^n$. Hence $\xi$ contains a patch at the origin which is
translationally congruent to 
the face with its $k$-neighborhood  
encoded by $\sigma_0$. Thus
$$a_{\Xi(\sigma)}(\xi)=
\left\{
\begin{array}{rl}
&a  \mbox{ if } \sigma=\sigma_0, \\
&0  \mbox{ if } \sigma\neq\sigma_0.
\end{array}
\right.
$$
Let $\gamma\in C^n(\Gamma_k, A)$. Then the above shows that  
$\f_k(\gamma)(\xi)=\gamma(\langle\sigma_0\rangle)$.  
Letting $\xi$ vary we see that $\f_k(\gamma)=0$ implies $\gamma=0$.
\qed\bigskip

\begin{lem}\label{surj} For all $\gamma\in \CC(\Xi^n,A)$ there is
  $k\in \N$ such that $\gamma\in \im (\f_k)$.\end{lem} 
\pn\textbf{\emph{Proof.}} By
Lemma~\ref{lem-topo} every element of  $\CC(\Xi^n,A)$
is a finite sum of $a_{\Xi(\sigma)}$, $a\in
A$ and $\sigma \in \Gamma_{k}$ for some $k$.
Since $a_{\Xi(\sigma)}\in \im(\f_{k})$ and $\im (\f_k) \subset
\im (\f_{k+1})$ the statement follows. 
\qed\bigskip

\begin{lem}\label{lem-com} The $\f_k$ intertwine the
  differentials, i.e.\ $\f_k\circ d_S = d_{PV}\circ \f_k$.
\end{lem} 
\pn\textbf{\emph{Proof.}} Recall the definition
 of the simplicial coboundary operator $d_S$
 on $\Gamma_k$: Given $\sigma\in S^n_k$ we have
$$d_{S}(a_\sigma)=\sum_{\tau\in S^{n+1}_k}\sum_{i=0}^{n+1}(-1)^i\delta_{\sigma\,\partial_i\tau} 
a_\tau. $$
Hence
$$\f_k\circ d_{S}(a_\sigma)(\xi)=\sum_{i=0}^{n+1}(-1)^i\delta_{\sigma\,\partial_i\sigma_\xi^k}$$ 
where $\sigma^k_\xi$ is the $k$-collared simplex of $\xi$ on $0$.
On the other hand
\begin{eqnarray*}
d_{PV}\circ \f_k(a_{\sigma}) & = & \sum_{\tau\in\mathcal{S}_0^{n+1}}\sum_{i=0}^{n+1} \theta_{\tau,i}
(a_{\Xi(\sigma)}) \\
& = & \sum_{\tau\in\mathcal{S}_0^{n+1}}\sum_{i=0}^{n+1} 
\imath\circ\t_{\tau,i}^*(a_{\Xi(\sigma)}
\left|_{\Xi(\partial_i\tau)}\right.) .\\
\end{eqnarray*}
Now 
\begin{eqnarray*}
\imath\circ\t_{\tau,i}^*(a_{\Xi(\sigma)}\left|_{\Xi(\partial_i\tau)}\right.)(\xi) &=& \left\{\begin{array}{ll}
a_{\Xi(\sigma)}(\t_{\tau,i}(\xi)) & \mbox{if } \xi\in\Xi(\tau) \\
0 & \mbox{else} \end{array}\right.\\ 
&=& \left\{\begin{array}{cl}
1 & \mbox{if } \sigma_{\xi-x_{\tau,i}}^k = \sigma\mbox{ and } 
\sigma_\xi^0 = \tau \\
0 & \mbox{else} \end{array}\right.
\end{eqnarray*}
and since $\sigma_{\xi-x_{\tau,i}}^k = \sigma$ and  
$\sigma_\xi^0 = \tau$ can happen iff
$\partial_i\sigma_\xi^k=\sigma$ the claim follows. \qed\bigskip

Recall that the cohomology of 
$\lim\limits_{\longrightarrow} C(\Gamma_k,A)$ is the direct limit 
$\lim\limits_{\longrightarrow} H(\Gamma_k,A)$ together with module maps 
$H(\pi_k):H(\Gamma_k,A) \to \lim\limits_{\longrightarrow}
H(\Gamma_k,A)$ such that $H(\pi_{k+1}) = H(\pi_k) \circ H(\pi_{k\,k+1})$.
\begin{thm}[\cite{BS}] \label{thm-BS}
$\f$ is a module isomorphism. It induces an isomorphism
in cohomology.
\end{thm}
\pn\textbf{\emph{Proof.}}
As the $\f_k$ intertwine the differentials they induce homomorphisms in
cohomology $H(\f_k) : H(\Gamma_k,A) \to H_{PV}(\Gamma_0,\CC_{lc}(\Xi,A))$
such that $H(\f_{k+1}) = H(\f_k) \circ H(\pi_{k\,k+1})$.
By the universal property we get a unique homomorphism which we denote
$H(\f)$ making
$$ \xymatrix{
            H(\Gamma_k,A) \ar@/_/[rdd]_{H(\f_k)} \ar[rr]^{H(\pi_{kj})}
            \ar[rd]^{H(\pi_k)} 
            && H(\Gamma_j,A) \ar@/^/[ldd]^{H(\f_j)} \ar[ld]_{H(\pi_j)} \\
             & \lim\limits_{\substack{ \longrightarrow \\
                 l}}{H(\Gamma_l,A)}  \ar[d]^{H(\f)} \\ &
             H_{PV}(\Gamma_0,\CC_{lc}(\Xi,A))} $$ 
commute.  
Hence $ H(\f)\circ H(\pi_k) ([\omega]) = H(\f_k)([\omega]) =
[\f_k(\omega)] = [\f\circ \pi_k(\omega)]$ where we have denote by
$[\cdot]$ cohomology classes.

 {\em Injectivity:}
Let $x \in \lim\limits_{\longrightarrow}{H(\Gamma_l,A)}$
such that $H(\f)(x) = 0$. There exists $k\in\N$ and a cochain in $\Gamma_k$
such that $x = H(\pi_k)[\omega]$. Thus $0 = [\f_k(\omega)]$.  Let
$\gamma\in \CC_{lc}(\Xi,A)$ satisfy
$\f_k(\omega)=d_{PV}\gamma$. 
By Lemma~\ref{surj}, there is $l\in \N$
and $\alpha\in C(\Gamma_l,A)$, such that $\gamma = \f_l(\alpha)$. By
Lemma~\ref{lem-com} 
$d_{PV}\f_l(\alpha)=\f_l(d_{S}\alpha)$. We may
suppose that $l\geq k$. Then the above implies that $\f_l(\omega) =
\f_l(d_S\alpha)$ and hence, by Lemma~\ref{inj} $\omega = d_S\alpha$.

{\em Surjectivity:} Let $[\gamma]\in H(\Gamma_0,\CC_{lc}(\Xi,A))$. Then
$\gamma = \f(\tilde\omega)$ for some $\tilde\omega \in
\lim\limits_{\longrightarrow}{C(\Gamma_l,A)}$. There exists $k\in\N$
and a cochain $\omega \in \Gamma_k$
such that $\tilde\omega = \pi_k[\omega]$. Hence $[\gamma] =
[\f\circ\pi_k(\omega)] = H(\f)\circ H(\pi_k)([\omega])$. \qed\bigskip

\subsection{PE simplicial as DL cohomology}

Recall the definition (\ref{eq-trho}) of 
$\tilde\rho_k : \R^d\to \Gamma_k$. In particular, to each simplex
$\sigma$ of $\p$ corresponds a simplex $\tilde\rho_k\circ\sigma$ of$\Gamma_k$.
Pulling these back over cochains we obtain module maps 
$\check \rho_k^*: C^*(\Gamma_k,A)\to C^*_{s-\p}(\p,A)$ 
$$ \check \rho_k^* (c)(\langle\sigma\rangle) =
c(\langle{\tilde\rho_k\circ\sigma}\rangle).$$  
They induce a new commuting diagram
$$ \xymatrix{
            C^n(\Gamma_k,A) \ar@/_/[rd]_{\check\rho_k^*} \ar[rr]^{\pi_{kj}}
            && C^n(\Gamma_j,A) \ar@/^/[ld]^{\check\rho_j^*}  \\
             & C_{s-\p}^n(\p,A) }$$
By the universal property of the direct limit there exists a
unique homomorphism $$\h:\lim\limits_{\substack{ \longrightarrow{} \\
l}}{C^n(\Gamma_l,A)} \longrightarrow C_{s-\p}^n(\p,A)$$ such that for all $k\in\N$,
$$ \h\circ \pi_k = \rho_k^*.$$
As Sadun explains in \cite{S2}, every s-PE cochain on $\p$ can be viewed
as the pull back of a cochain on $\Gamma_k$ provided $k$ is
sufficiently large and vice versa. In other words, the analogue of Lemma~\ref{inj} and Lemma~\ref{surj} are true and
$\h$ is an isomorphism. 
It is trivial that the $ \rho_k^*$ intertwine the differentials as they are defined in essentially the same way.
Hence the same proof as that for Thm.~\ref{thm-BS} yields that
 $\h$ induces an isomorphism between DL cohomology and s-PE simplicial cohomology.

\subsection{PE simplicial and PV-cohomology}
The last two sections can be summarized by saying that the strong
versions of PE simplicial  and PV-cohomology are abstractly the same
thing: in both cases the cochains arise as models for the direct limit
of the cochains of APG-complexes. 
In fact, by the universal property the module morphism
$\PEPV : C_{s-\p}(\p,A) \to \CC_{lc}(\Xi,A)$:
$$ \PEPV = \f\circ (\check\rho^*)^{-1}$$
is an isomorphism which, by Lemma~\ref{lem-com} and its analogon intertwines
the differentials. It hence induces an isomorphism between 
the strong versions of PE simplicial and PV cohomology.
By construction $\PEPV(\check\rho^*_k(a_\sigma)) =
\f_k(a_{\tilde\sigma}) = 
a_{\Xi(\tilde\sigma)}$ and hence $ \PEPV(c)$ is 
on the dense orbit through $\p$ given by
\begin{equation}\label{eq-PEPV}
 \PEPV(c)(\p-p) = c(\langle{\sigma_p}\rangle)
 \end{equation}
where $\sigma_{p}$ is the simplex of $\p$ whose puncture is on $p$.

\begin{lem}\label{cont} $\PEPV$ is an isometry and thus extends to an
  isomorphism, which we also denote by $\PEPV$, 
between  $C_{w-\p}(\p,A)$ and $\CC(\Xi,A)$.
\end{lem}
\pn\textbf{\emph{Proof.}} That $\PEPV$ is an isometry is direct. 
The continuous extension of a bijective isometry is bijective as well.
\qed\bigskip

\begin{cor}\label{cor-PEPV}
$\PEPV$ induces a isomorphism 
between the weakly
pattern equivariant cohomology $H_{w-P}(\p,A)$ and the weak 
PV-cohomology $H_{PV}(\Gamma_0,\CC(\Xi,A))$.
\end{cor}
\begin{proof} By continuity $\PEPV$ intertwines the
  differentials. Hence $\PEPV$ is a bijective chain map. 
\end{proof}

\section{PE de Rham versus PE simplicial cohomology}
\subsection{A PE de Rahm theorem}
The de Rham theorem for smooth manifolds states that, given a smooth
triangulation of the manifold and hence a simplicial decomposition,
then the de Rham cohomology of the manifold and
the simplicial cohomology of the complex are isomorphic. Moreover, the
isomorphism has a very simple form, it is induced from a chain map
which associates to a $k$-form the $k$-cochain whose
evaluation on the generator associated with a $k$-simplex is given by
integrating the form over the simplex. A very careful exposition of
the proof is given in \cite{ST} for the case of finite simplicial
complexes in euclidean spaces.

We shall adapt this proof to obtain a PE version of the de Rham
theorem the role of the manifold 
being played by $\rn$ with its simplicial decomposition defined by the
tiling $\p$. 
This then will lead to an explicit isomorphism between the s-PE de Rham
cohomology of $\rn$ and the s-PE simplicial cohomology of $\p$ with
values in $\R$.
We furthermore show that the maps involved are continuous so that the
isomorphism extends to an explicit isomorphism between the weak
versions.

Recall that $\p$ defines a smooth triangulation
on $\rn$ but even a smooth triangulation. This allows us to integrate a
$k$-form $\omega\in\A_b^k(\rn)$ over a $k$-simplex $\sigma$. In fact
we write $\int_\sigma \omega$ to mean $\pm$ the integral of $\omega$
over the set $\im  \sigma $ where the sign is $+$ provided
the orientation of $\sigma$ is the same than that induced on
$\im(\sigma)$ by the (chosen) orientation of $\rn$, and $-$ otherwise.
We equip $C_b(\p,\R)$ with the sup-norm: $\|c\| = \sup_\sigma
|c(\langle \sigma\rangle)|$. This makes it a complete real vector
space on which the simplicial differential acts continuously. 
\begin{df}
Let $J_l:\A_b^l(\rn)\to C_b^l(\p,\R)$ be given by 
$$J_l(\omega)(\langle \sigma\rangle) = \int_\sigma \omega$$
for any $l$-simplex $\sigma$.
\end{df}
Our aim is to prove the following theorem
\begin{thm}~\label{thm-deRham} 
$J_l$ induces an isomorphism between $H^l_{s-\p}(\rn)$ and $H^l_{s-\p}(\p,\R)$ 
and also between $H^l_{w-\p}(\rn)$ and $H^l_{w-\p}(\p,\R)$.
\end{thm}
We should say that the first part could also be derived from the de
Rham type theorems for transversally locally constant tangential forms
\cite{KP} or for branched manifolds \cite{S2}, but these proofs use  
the \v{C}ech-de Rham double complex and hence yield a priori isomorphisms
with \v{C}ech cohomology. 
We provide here a proof avoiding \v{C}ech cohomology and double complexes
at all. So it is more elementary but also longer. But it has the
advantage of extending directly to the topological closures (weak
versions).    
\begin{lem} \label{lem-J}
$J_l$ has the following properties:
\begin{enumerate}
\item $J_{l+1}\circ d = d_S \circ J_l$.
\item $J_l$ is continuous.
\item 
If $\omega$ is s-PE then $J_l(\omega)$ is s-PE.
\item 
If $\omega$ is w-PE then $J_l(\omega)$ is w-PE.
\end{enumerate}
\end{lem}
\pn\textbf{\emph{Proof.}}
The first statement is Stokes theorem. Continuity follows from
$\|J_l(\omega)\|\leq \sup_\sigma
|\int_\sigma \omega|\leq \|\omega\|_\infty\sup_\sigma vol(\sigma)$.
The third statement is evident and implies the last one by
 continuity of $J_l$.
\qed\bigskip

The first step in showing a de Rham theorem is the construction of right
inverse maps for $J_l$. We proceed
as in \cite{ST}. 

\begin{lem} \label{lem-al}
There are module maps $\alpha_l:C_b^l(\p,\R)\to \A_b^l(\rn)$ which
satisfy the following properties: 
\begin{enumerate}
\item $J_l\circ\alpha_l = id$.
\item $\alpha_{l+1}\circ d_S = d \circ \alpha_l$.
\item $\alpha_l$ is continuous.
\item 
If $c$ is s-PE then $\alpha_l(c)$ is s-PE.
\item 
If $c$ is w-PE then $\alpha_l(c)$ is w-PE.
\end{enumerate}
\end{lem}
\pn\textbf{\emph{Proof.}}
We follow the construction of $\alpha_l$ given in \cite{ST}. 
It is based on a partition of
unity $\{g_v\}_{v\in\p^{(0)}}$ which is subordinate to the covering given
by the stars of the vertices $v$: The star $St(v)$ of $v$ is simply the
interior of the 
$1$-neighborhood of $v$, i.e.\ the union over all open simplices
touching $v$. Let $F_v$ be the compact set points of $St(v)$ which have at least
distance $\frac1{n+1}$ from the boundary of $St(v)$ and $G_v$ the
closed set of points which have at most distance $\frac1{n+2}$ from
the complement of $St(v)$. As is well known there exists a smooth positive
function $\check g_v$ which vanishes on $G_v$ and equals $1$ on
$F_v$. Clearly we 
can require that the choice of  $\check g_v$ be made so that if $v$ and
$v'$ have the same $1$-neighborhood up to translation then 
$\check g_v$ and  $\check g_{v'}$ coincide up to translation by $v-v'$.
Now $g_v = \check g_v / \sum_{v'\in\p^{(0)}}\check g_{v'}$ and
$\alpha_l$ is given as follows. Let $\sigma$
be a $l$-simplex  and $v_i$ the $i$th vertex, i.e.\ the image under
$\sigma$ of $x_i\in\Delta^l$. Then 
$$ \alpha_l(1_\sigma)=\omega_\sigma := l!\sum_{i=0}^l (-1)^i g_{v_i}
dg_{v_0}\wedge\cdots \widehat{dg_{v_i}}\cdots dg_{v_l}.$$
A proof of the first two properties can be found in \cite{ST}.
The above formula shows that $\alpha_l(1_\sigma) =
{\t^x}^*\alpha_l(1_{\sigma'})$
(${\t^x}^*$ denotes the pull back of ${\t^x}$)
if $\sigma={\t^x}(\sigma')$ and their $2$-neighborhoods coincide.
 $\alpha_l(c)$ is therefore s-PE provided $c$ is a s-PE cochain.

To show continuity let $c= \sum_{\sigma\in\p^{(l)}} r_\sigma1_\sigma$,
$r_\sigma\in\R$. Then
$$ s_k (\alpha_l(c)) \leq \sum_{|\kappa|\leq k}\|D^\kappa \sum_{\sigma\in\p^{(l)}}
  r_\sigma  \omega_\sigma\|_\infty. $$
Since the support of $\omega_\sigma$ is contained in $St(v)$ and the
number of $l$-simplices in a star is uniformly bounded, let's say by
$N_l$, we have $\sup_x |D^\kappa \sum_{\sigma\in\p^{(l)}}
  r_\sigma  \omega_\sigma(x)|\leq N_l \sup_{\sigma\in\p^{(l)}}|r_\sigma|
  \sup_x |D^\kappa\omega_\sigma(x)|$ and hence 
$$ s_k (\alpha_l(c))\leq N_l\|c\| \max_{\sigma\in\p^{(l)}} s_k(\omega_\sigma).$$
Hence $\alpha_l$ is continuous from which the last statement follows
by continuous extension.  
\qed\bigskip

\begin{cor} $J_l$ induces a surjective homomorphism between
  $H^l_{s-\p}(\rn)$ and $H^l_{s-\p}(\p,\R)$  
and also between $H^l_{w-\p}(\rn)$ and $H^l_{w-\p}(\p,\R)$.
\end{cor}
\newcommand{\eps}{\epsilon}
\newcommand{\ZZ}{\mathcal Z}
The proof injectivity of these maps  is essentially based on
Poincar\'e's lemma.  
In contrast to the usual case we need however good control on the
contracting homotopies involved. Let $U$ be an open star-shaped subset
of $\rn$ and $u\in U$ a center, i.e.\ for all $x\in U$ we have
$\forall 0\leq t \leq 1: tx + (1-t)u \in U$. Then we denote by
$h_{k-1}^{U,u}:\A^k_b(U)\to \A^{k-1}_b(U)$, $k>0$, the following module map: 
$$h_{k-1}^{U,u}(f dx_{i_1}\wedge\cdots dx_{i_k})(x) = I^u_{k-1}(f)(x)
\sum_{j=1}^k (-1)^j (x_{i_j}\!-\!u_{i_j}) dx_{i_1}\!\wedge\cdots
\widehat{dx_{i_j}}\!\cdots dx_{i_k}$$
where
$$I^u_{k-1}(f)(x)= \int_0^1 t^{k-1} f(tx+(1-t)u)dt.$$
\begin{lem}\label{lem-h}
Let $U$ be bounded and as above. $h_k^{U,u}$ have the following properties
\begin{enumerate}
\item $h_k^{U,u} \circ d = d \circ h_{k-1}^{U,u}$ for $k\geq 1$,
\item $h_k^{U-x,u-x} = {\t^x}^* \circ h_k^{U,u}$,
\item $h_k^{U,u}$ is continuous.
\end{enumerate}
\end{lem}
\pn\textbf{\emph{Proof.}}
$h_{k-1}^{U,u}$ is a standard homotopy for $U$ and the first statement
(Poincar\'e lemma for $U$)
proven for instance in \cite{ST}. The second statement is a direct
consequence of the definition.
As for the third, continuity of $h_{k-1}^{U,u}$ follows from continuity
of $I^u_{k-1}$ since the form
$\omega(x) = \sum_{j=1}^k(-1)^j (x_{i_j}-u_{i_j}) dx_{i_1}\wedge\cdots
\widehat{dx_{i_j}}\cdots dx_{i_k}$ does not depend on $f$ and has
bounded seminorms given that $U$ is bounded. Now
$$\| D^\kappa I^u_{k-1}(f)\|_\infty = \left\|\int_0^1 t^{k-1+|\kappa|}
  (D^\kappa f)(tx+(1-t)u)dt \right\|_\infty \leq \|D^\kappa f\|_\infty$$
shows that $I^u_{k-1}$ is indeed continuous.
\qed\bigskip

Let $\sigma$ be an $k$-simplex. 
We denote by $[\sigma]_\eps$ the $\epsilon$-neighborhood 
of $\im(\sigma)$ and  by $[\partial\sigma]_\eps$ 
the $\epsilon$-neighborhood of its boundary $\partial\im(\sigma)$. 
We denote by 
$\ZZ^r_b(U)$ the closed bounded $r$-forms on $U$.  
For $r\neq k$, $r,k\geq 1$ let 
$$\D_\epsilon^r(\sigma) = \{(\omega,\tau)\in \ZZ_b^r([\sigma]_\eps)  \oplus \A_b^{r-1}([\partial \sigma]_\eps): \omega\left|_{[\partial \sigma]_\eps}\right. = d\tau\} $$
and for $r=k\geq 1$ let 
$$\D_\epsilon^k(\sigma) = \{(\omega,\tau)\in \ZZ_b^k([\sigma]_\eps)  \oplus \A_b^{k-1}([\partial \sigma]_\eps): \omega\left|_{[\partial \sigma]_\eps}\right. = d\tau, \int_\sigma \omega = \int_{\partial \sigma}\tau\} .$$
\begin{lem}
There exists an $\eps>0$ and a module map 
$\beta^r(\sigma):\D_\epsilon^r(\sigma)\to \A_b^{r-1}([\sigma]_\eps)$ satisfying
\begin{enumerate}
\item $\tau' = \beta^r(\sigma)(\omega,\tau)$ coincides with $\tau$ on
  $[\partial \sigma]_\eps$, 
\item $d\tau' = \omega$,
\item If ${\t^x}(\sigma)$ is a simplex then $\beta^r({\t^x}(\sigma))=
  {\t^x}^*\circ\beta^r(\sigma)$, 
\item $\beta^r(\sigma)$ is continuous.
\end{enumerate}
\end{lem}  
\pn\textbf{\emph{Proof.}}
\newcommand{\Dd}{\E}
The proof is by induction. For that we need to consider a second
module map.
For $r\neq k$, $r\geq 0$, $k\geq 1$ let 
$\Dd_\eps^r(\sigma) =  \ZZ_b^r([\partial \sigma]_\eps$
and for $r=k-1\geq 0$ let 
$$\Dd_\eps^k(\sigma) = \{\omega\in \ZZ_b^k([\partial \sigma]_\eps : 
\int_{\partial \sigma} \omega = 0 \} .$$
We aim to show that there exists a module map
$\gamma^r(\sigma):\Dd_\eps^r(\sigma)\to \ZZ_b^{r}([\sigma]_\eps)$ which
satisfies 
\begin{enumerate}
\item $\gamma^r(\sigma)(\omega)$ coincides with $\omega$ on $[\partial
  \sigma]_\frac{\eps}2$, 
\item If ${\t^x}(\sigma)$ is a simplex then $\gamma^r({\t^x}(\sigma))=
  {\t^x}^*\circ\gamma^r(\sigma)$, 
\item $\gamma^r(\sigma)$ is continuous.
\end{enumerate}
\begin{itemize}
\item[$\gamma^0(\sigma)$] Let $\omega\in \Dd_\eps^0(\sigma)$. Since
  $d\omega = 0$, $\omega$ is constant on any connected subset of
  $\partial \sigma$. If $k\neq 1$  $\partial \sigma$ is connected and
  hence $\omega$ constant. 
If $k=1$ the extra condition  $\int_\sigma \omega = 0$ implies that
$\omega$ is constant. Thus for all choices of $k$   
we can define $\gamma^0(\sigma)(\omega)$ to be the 
constant extension of $\omega$ to $\sigma$. The properties stated are
obviously satisfied. 
\item[$\beta^r(\sigma)$] We assume the existence of
  $\gamma^{r-1}(\sigma)$ satisfying the above properties. This is
  certainly correct for $r=1$ and then will follow inductively in view
  of the next step. Let $(\omega,\tau)\in \D_\eps^r(\sigma)$. Let
  $b_\sigma$ be the barycenter of $\sigma$ and
  $h^\sigma=h^{[\sigma]_\eps,b_\sigma}$. Define 
$$ \beta^r(\sigma)(\omega,\tau) := h^\sigma(\omega)
-\gamma^{r-1}(\sigma)(h^\sigma(\omega)\left|_{[\partial \sigma]_\eps}
\right.-\tau)  .$$ 
The first two properties are verified in \cite{ST}. The other two
follow immediately from the properties of  
$\gamma^{r-1}(\sigma)$ and Lemma~\ref{lem-h}.
\item[$\gamma^r(\sigma)$] We assume the existence of
  $\beta^{r}(\sigma)$ satisfying the properties stated in the lemma. 
This follows inductively in view of the preceding step.
Let $\omega\in \Dd_\eps^r(\sigma)$. 
Let $v$ be the first vertex of $\sigma$ (w.r.t.\ the ordering of the
vertices) and $\sigma' $ the corresponding $k-1$-simplex
of the boundary of $\sigma$, i.e.\ the face which does not contain $v$.
Then $[\partial \sigma \backslash
\sigma']_\eps$ is star-shaped with center $v$. 
We let $\mu'_\sigma=h^{[\partial \sigma \backslash
\sigma']_\eps,v}(\omega\left|_{[\partial \sigma \backslash
\sigma']_\eps]}\right.)$ and $\mu''_\sigma =
\beta^r(\sigma)(\omega\left|_{[\sigma']_\eps}\right. , \mu_0\left|_{[\partial
    \sigma']_\eps}\right.)$. 
$\mu'_\sigma$ and $\mu''_\sigma$ coincide on their common domain
and hence define 
an $r-1$-form $\mu_\sigma$ on $[\partial \sigma]_\eps$. 
It follows from Lemmata~\ref{lem-simpl} and \ref{lem-h} and the induction
hypothesis for $\beta^r$ 
that $\mu_{\t^x(\sigma)}=
  {\t^x}^*\mu_\sigma$ provided $\t^x(\sigma)$ is a simplex of the tiling. 
Now consider a family of smooth functions 
$b_\sigma:[\sigma]_\eps\to \R$, defined for all $k$-simplices, 
which is s-PE in the sense that $b_{\t^x(\sigma)} = {\t^x}^*b_\sigma$. 
We suppose that each $b_\sigma$ is identically $1$ on
$[\partial\sigma]_\frac\eps2$ and identically $0$ on the complement of
$[\partial\sigma]_\eps$. Then 
$$ \gamma^r(\sigma)(\omega) := d(b_\sigma \mu_\sigma)$$
satisfies $\gamma^r(\sigma)(\omega)=\omega$ on 
$[\partial\sigma]_\frac\eps2$, as is shown in \cite{ST}, and
$\gamma^r({\t^x}(\sigma))= {\t^x}^*\circ\gamma^r(\sigma)$.
Continuity of  $ \gamma^r(\sigma)$ follows from the continuity of 
$h^{[\partial \sigma \backslash\sigma']_\eps,v}$, $\beta^r(\sigma)$
and $d$, and smoothness of $b_\sigma$. 
\end{itemize}
\qed\bigskip

Let $\D_\epsilon^r(\p^{(k)})$ be defined as
$\D_\epsilon^r(\sigma)$ except replacing the 
simplex by the $k$-skeleton. If $r=k$ we demand the integral equation
$\int_\sigma \omega = \int_{\partial \sigma}\tau$ to hold for all
$k$-simplices. 
We then define 
$\beta^r:\D_\epsilon^r(\p^{(k)})\to \A_b^{r-1}([\p^{k)}]_\eps)$ by glueing the
$\beta^r(\sigma)(\omega\left|_{[\sigma]_\eps}\right.,\tau\left|_{[\partial
    s]_\eps}\right.)$ defined for the $k$-simplices $\sigma$ of $\p^{(k)}$
together which is possible as they agree with
$(\omega,\tau)$ along the $\epsilon$-neighborhood of $\p^{(k-1)}$.  
\begin{cor}\label{cor-3.7}
\begin{enumerate}
\item If $(\omega,\tau)\in \D_\epsilon^r(\p^{(k)})$ is s-PE then
  $\beta^r(\omega,\tau)$ is s-PE. 
\item If $(\omega,\tau)\in \D_\epsilon^r(\p^{(k)})$ is w-PE then
  $\beta^r(\omega,\tau)$ is w-PE. 
\end{enumerate}
\end{cor}
\begin{lem} Let $\omega\in\ZZ^l_b(\rn)$, $l\geq 1$, and $J_l(\omega) = d_S c$. 
Then there exists $\tau\in \A_b^{l-1}(\rn)$ such that $\omega = d\tau$.
Moreover,
\begin{enumerate}
\item If $\omega$ and $c$ are s-PE then  $\tau$ can be chosen s-PE.
\item If $\omega$ and $c$ are w-PE then  $\tau$ can be chosen w-PE.
\end{enumerate}
\end{lem}
\pn\textbf{\emph{Proof.}}
We follow again \cite{ST} constructing $\tau$ by a composition of
module maps
$$\tau = \tilde\beta_n\circ\cdots \circ\tilde\beta_0(\omega).$$
$\tilde\beta_0(\omega)$ is an $l-1$-form on the $\eps$-neighborhood
of $\p^{(0)}$. We may therefore define it through its restrictions to the
connected components of $[\p^{(0)}]_\eps$, 
namely the restriction to $[v]\epsilon$ is given by 
$$h^{[v]\epsilon}(\omega)  - \delta_{l
  1}(h^{[v]\epsilon}(\omega)(v)-c(v)).$$ 
If $k>0$ then 
$$\tilde\beta_{k}(\omega)=\beta^l(\omega,\tilde\beta_{k-1}(\omega))
- \delta_{l-1\,
  k}\alpha_{l-1}(J_{l-1}(\beta^l(\omega,\tilde\beta_{k-1}(\omega))-c).$$ 
This construction of $\tau$ follows exactly the lines of the proof of
\cite{ST} and so the proof given there for the first statement applies
also in our context.  
The two remaining statements are now obtained by finite iteration of the results of Lemmata~\ref{lem-J}, \ref{lem-al}, \ref{lem-h} and Cor.~\ref{cor-3.7}. 
\qed\bigskip

We thus have proved Theorem~\ref{thm-deRham}.

\begin{cor}
The map $\iso= \PEPV\circ J: \A_{w-\p}(\R^d)\to \CC(\Xi,\R)$ 
induces an isomorphism $H(\iso)$ between $H_{w-\p}(\R^d)$ and
$H_{PV}(\Gamma_0,\CC(\Xi,\R))$ which
restricts to an isomorphism between $H_{s-\p}(\R^d)$ and
$H_{PV}(\Gamma_0,\CC_{lc}(\Xi,\R))$. 

Let $\omega\in\A_{w-\p}^n(\R^d)$. An explicit formula for
$\iso(\omega) \in \CC(\Xi^n,\R)$  
is given by
$$ \iso(\omega)( \p - p ) = \int_{\sigma_{p}}\omega, $$
where $\sigma_{p}$ is the simplex of $\p$ whose puncture is on
$p\in\p^{n,punc}$. \end{cor}
\pn\textbf{\emph{Proof.}} Combine Cor.~\ref{cor-PEPV} with
Thm.~\ref{thm-deRham}  and Eq.~\ref{eq-PEPV}.
\qed

\section{Some applications}

\subsection{Weak cohomology of cut \& project tilings with dimension and codimension equal to one}
We saw that the weak PE de Rham cohomology, the weak PE simplicial
cohomology with real values and the weak PV cohomology with real
values are all isomorphic. We refer to this cohomology simply as the
weak tiling cohomology. The following calculation shows that this
cohomology is infinitely generated even for the simplest aperiodic tilings of
finite local complexity. Our calculation is based on an old result, namely  
 that the transversally {\em continuous} tangential cohomology of a Kronecker
foliation on a torus is infinitely generated in degree $1$ \cite{Mostow76}. 
We expect that weak  tiling cohomology is always infinitely generated
 for aperiodic tilings of finite local complexity. 

As explained in \cite{BS} for one dimensional tilings PV-cohomology is just the
cohomology of the group $\Z$ with coefficients in $\CC(\Xi^{0},\Z)$
where the action is induced by shifting the tiling by a tile (this is
more precisely the first return map into the $0$-dimensional 
$\Delta$-transversal
$\Xi^{0}$ of the $\R$-action by translation of the tilings). 
The same argument applies to weak PV-cohomology upon replacing
$\CC(\Xi^{0},\Z)$ by $\CC(\Xi^{0},\R)$. 

For cut \& project tilings with dimension and codimension equal to one 
the dynamical system $(\Xi^{0},\Z)$ 
can be described as follows \cite{FHKmem}: 
The  cut \& project tiling is determined by two data, $\theta$ and $K$.
$\theta$ is an irrational number defining
the slope of a line $E\subset \R^2$ w.r.t.\ a choice of orthonormal
base. $K$ is a subset of the orthocomplement of $E$ which
is supposed to be compact and the closure of its interior. 
These data define a one dimensional tiling 
of $E\cong\R$. Indeed the tiles are intervals whose end points are 
given by the 
ortho-projection onto $E$ of all points in $\Z^2\cap (K+E)$ ($\Z^2$ is
the lattice generated by the base and it is supposed that no point of
that intersection lies on the boundary of $K+E$). Since $\theta$ is
irrational the tiling is aperiodic. 

On the $0$-dimensional $\Delta$-transversal of this tiling $\Xi^0$
the group $\Z$ acts via the first return map 
of the $\R$-action by translation of the tilings.  
This dynamical system $(\Xi^{0},\Z)$ factors onto     
the dynamical system $(S^1=\R/\Z,\Z)$  given by the rotation around $\theta$.
The factor map $\mu:\Xi^{0}\to S^1$ is almost one-to-one. The points
where it is not one-to-one, the preimages of so-called cut points, 
form a union of orbits under the $\Z$-action. 
On these orbits $\mu$ is two-to-one  
and we denote such points by $a=(a^+,a^-)$.
\begin{thm} The first weak tiling
cohomology of a cut \& project tiling with dimension and codimension equal to one is infinitely generated.\end{thm}
\begin{proof} By the last theorem and the above remark the first weak tiling cohomology is isomorphic to
$H^1(\Z,\CC(\Xi^{0},\R))$, the cohomology of the group $\Z$ with values in $\CC(\Xi^{0},\R)$. 
The pull back of the factor map $\mu:\Xi^{0}\to S^1$ gives rise to a $\Z$-equivariant short exact sequence of $\Z$-modules 
$$ 0\to \CC(S^1,\R) \stackrel{\mu^*}\longrightarrow \CC(\Xi^{0},\R) \to Q\to 0$$ where $Q=\CC(\Xi^{0},\R)/\CC(S^1,\R)$
naturally inherits a $\Z$-action.
We claim that $H^0(\Z,Q)=0$ so that 
the long exact sequence in group cohomology contains the short exact sequence
$$ 0 \to H^1(\Z, \CC(S^1,\R)) \to H^1(\Z, \CC(\Xi^{0},\R)) \to 
H^1(\Z, Q)\to 0.$$
We prove the claim:
For a pre-image of a cut point $a=(a^+,a^-)\in \Xi^{0}$ let
$\tau_a:\CC(\Xi^{0},\R)\to \R$ be given by $\tau_a(f) = f(a^+)
-f(a^-)$. 
This is a module homomorphism which vanishes on functions of 
$\mu^*(\CC(S^1,\R))$ 
and hence extends to a homomorphism $\tau_a:Q\to\R$. $H^0(\Z,Q)$ is
the group of invariant elements. 
Assuming that $H^0(\Z,Q)$ is non-trivial
suppose that $f\in \CC(\Xi^{0},\R)$ projects onto a non-trivial invariant element in $Q$. Then there exists an $a$ such that $\tau_a(f)=\tau_a(f\circ\varphi)\neq 0$. Choose $0<4 \epsilon<|\tau_a(f)|$. There exists $g\in \CC_{lc}(\Xi^{0},\R)$ such that $\|f-g\|_\infty<\epsilon$. 
This implies that $|\tau_a(f\circ\varphi^n)-\tau_a(g\circ\varphi^n)|<2\epsilon$, for all $n\in\Z$. Hence     $|\tau_a(f)-\tau_a(g\circ\varphi^n)|<2\epsilon<\frac{1}{2}|\tau_a(f)|$, for all $n\in\Z$, and therefore  $|\tau_a(g\circ\varphi^n)|>\frac{1}{2}|\tau_a(f)|$, for all $n\in\Z$.
It follows that $g$ has infinitely many jumps. But by compactness of
$\Xi^{0}$ the elements of $\CC_{lc}(\Xi^{0},\R)$ have only finitely
many jumps. This being a contradiction we must have
$H^0(\Z,Q)=0$. 

Hence $H^1(\Z,\CC(\Xi^{0},\R))$ is an extension of $H^1(\Z,Q)$ by
$H^1(\Z,\CC(S^1,\R))$. 
Now by the results of Mostow \cite{Mostow76}[page 98f], $H^1(\Z,\CC(S^1,\R))$ is
infinitely generated, as $\theta$ is irrational.
\end{proof}

\subsection{The mixed group in PV-cohomology}

There are two approaches to the deformation theory of tilings. The first
is based on the description of tiling cohomology as the direct limit 
$\lim_{\to}H((\Gamma_l,\R^d)$ and asymptotic negligible cocycles  \cite{CS}, and the
second on pattern equivariant 1-forms \cite{K2}. In the context of the second approach a 
deformation of a Delone set $\p$ into a Delone set $\p'$ is given by a bi-Lipschitz smooth function
$\varphi:\R^d\longrightarrow\R^d$ which has a strongly pattern
equivariant differential and satisfies $\p'=\varphi(\p)$. Moreover, if two such functions differ on $\p$ by a constant vector  (which amounts to a global translation) then they define the same deformation. 
The deformations of $\p$ are therefore parameterized by the
elements of $Z_{s-\p}^1(\R^d,\R^d)/N_{s-\p}^1(\R^d,\R^d)$ where
$$N_{s-\p}^1(\R^d,\R^d)=\{d\varphi\in B_{s-\p}^1(\R^d,\R^d)| \varphi \mbox{ is
constant on }\p \}.$$ 
Deformations of tilings can be seen as deformations of their set of punctures.

Denote by $\mathcal{B}_{||.||_\infty}(\omega,\varepsilon)$ the open
$\varepsilon$-ball around an element $\omega\in
Z_{s-\p}^1(\R^d,\R^d)$ w.r.t. the uniform norm. For sufficiently small $\epsilon$ the elements of
$Z_{s-\p}^1(\R^d,\R^d)\cap
\mathcal{B}_{||.||_\infty}(d\mbox{id},\varepsilon)$ define
invertible deformations (in sense of \cite{K2}) of $\p$.
If two such elements differ by an element of
$B_{s-\p}^1(\R^d,\R^d)$ then their images on $\p$ are mutually locally derivable.
If their difference is in $B_{w-\p}^1(\R^d,\R^d)$ then their images on $\p$ define pointed topological conjugate dynamical systems. The elements
near the class of $d$id in the mixed quotient group
$$Z_{s-\p}^1(\R^d,\R^d)/B_{w-\p}^1(\R^d,\R^d)\cap
Z_{s-\p}^1(\R^d,\R^d)$$ parameterize therefore the set of  small deformations modulo
bounded deformations which are in the same pointed conjugacy class.

The following corollary yields a description of  
the parameter space of deformations modulo topological conjugacy in terms of a mixed PV-cohomology group. We denote by $Z_{PV}^*(\Gamma_0,-)$ and $B_{PV}^*(\Gamma_0,-)$ PV coycles and PV coboundaries, respectively.
\begin{cor} The pattern equivariant mixed quotient 
group $Z_{s-\p}^*(\R^d,\R^d)/B_{w-\p}^*(\R^d,\R^d)\cap
Z_{s-\p}^*(\R^d,\R^d)$ is isomorphic to
$$Z_{PV}^*(\Gamma_0,\CC_{lc}(\Xi,\R^d))/B_{PV}^*(\Gamma_0,\CC(\Xi,\R^d))\cap
Z_{PV}^*(\Gamma_0,\CC_{lc}(\Xi,\R^d)).$$
\end{cor}
\pn\textbf{\emph{Proof.}} $H^*(\iso):H_{w-\p}^*(\R^d)\to
H_{PV}^*(\Gamma_o,\CC(\Xi,\R))$ 
is injective and hence its restriction to 
$Z_{s-\p}^*(\R^d)/B_{w-\p}^*(\R^d)\cap Z_{s-\p}^*(\R^d)\subset H_{w-\p}^*(\R^d)$ is injective as well.
$H^*(\iso):H_{s-\p}^*(\R^d)\to H_{PV}^*(\Gamma_o,\CC_{lc}(\Xi,\R))$ is surjective and $\iso(B_{w,\p}^*(\R^d))\subset B_{PV}^*(\Gamma_0,\CC(\Xi,\R))$. It follows that the restriction of $H^*(\iso)$ to 
$Z_{s-\p}^*(\R^d)/B_{w-\p}^*(\R^d)\cap Z_{s-\p}^*(\R^d)$ is surjective.

\qed

\end{document}